\documentclass[12pt,leqno]{amsart}
\newtheorem{thm}{Theorem}[section]
\newtheorem{defi}{Definition}[section]
\newtheorem{lem}{Lemma}[section]

\newcommand{\be}{\begin{equation}}
\newcommand{\ee}{\end{equation}}
\numberwithin{equation}{section}
\newcommand{\bea}{\begin{eqnarray}}
\newcommand{\eea}{\end{eqnarray}}
\newcommand{\beb}{\begin{eqnarray*}}
\newcommand{\eeb}{\end{eqnarray*}}
\usepackage{amssymb,amsfonts,amsthm,setspace,indentfirst}
\usepackage[dvips]{graphics}
\usepackage{epsfig}
\begin{document}
\title{Rough convergence of sequences in a $\textit{S}$-metric space}
\author{Rahul Mondal$^{1}$ and Sukila khatun$^{2}$}
\address{$^{1}$Department of Mathematics, Vivekananda Satavarshiki Mahavidyalaya, Manikpara, Jhargram -721513, West Bengal, India.} 
\address{$^{2}$Department of Mathematics, The University of Burdwan,
Golapbag, Burdwan-713104, West Bengal, India.}
\email{$^{1}$imondalrahul@gmail.com}
\email{$^{2}$sukila610@gmail.com}
\begin{abstract}
Phu\cite{PHU} introduced the idea of rough convergence of sequences in a normed linear space. Here using the idea of Phu we have brought the idea of rough convergence of sequences in a $S$-metric space and discussed some of its basic properties.
\end{abstract}
\noindent\footnotetext{$\mathbf{2010}$\hspace{5pt}AMS\; Subject\; Classification: 40A05, 40A99.\\
{Key words and phrases: Rough convergence, rough Cauchy sequence, $S$-metric, Rough limit point.}}
\maketitle
\section{\bf{Introduction}}
\noindent The abstract formulation of the concept of distance between two points of an arbitrary non empty set has been seen in a metric space. It is well-known to all that by generalizing the concept of distance the central concepts of real and complex analysis like  open and closed sets, convergence of a sequence, continuity and uniform continuity etc. has been generalized in a metric space.\\ 
\indent Later many authors have found out many generalizations of the concept of metric space. The idea of a 2-metric has been given by Gahler\cite{Gahler} in 1963 and in 1992 Dhage\cite{Dhage} introduced the idea of $D$-metric as a generalization of ordinary metric. Later in 2006 we have got another generalization of the concept of metric space by Mustafa and Sims\cite{Mustafa} to bring some new ideas in the fixed point theory. Huang and Zhang \cite{HLG} introduced the idea of cone metric spaces in 2007 generalizing the concept of distance  in an ordinary metric space. The idea of a $S$-metric space has been given by S. Sedghi et al in 2012 for generalization of fixed point theorems in this new structure. \\
\indent The idea of rough convergence of sequences has been firstly given by Phu\cite{PHU} in a normed linear space. There he also introduces the idea of rough Cauchy sequences. Many works\cite{AYTER2, PMROUGH1, PMROUGH2} have been done as an extension of these ideas. Banerjee and Mondal \cite{RMROUGH} studied the idea of rough convergence of sequences in a cone metric space. Recently the idea of rough Cauchy sequences has been introduced by Mondal\cite{RM} in a cone metric space. Here we have discussed the idea of rough convergence of sequences in a $S$-mertric space.
\section{\bf{Preliminaries}}\label{preli}
\begin{defi}\cite{PHU}
Let $\{ x_{n} \}$ be a sequence in a normed linear space $(X, \left\| . \right\|)$,  and $r$ be a non negative real number. Then $\{ x_{n} \}$ is said to be $r$-convergent to $x$ if for any $\epsilon >0$, there exists a natural number $k$ such that $\left\|x_{n} - x \right\| < r +\epsilon$ for all $n \geq k$.
\end{defi}
\begin{defi}\cite{smetric}
In a nonempty set $X$ a function $S: X^{3} \longrightarrow [0, \infty)$ is said to be a $S$-metric on $X$ if the following three conditions hold, for every $x,y,z,a \in X$:\\
$(i)$ $S(x,y,z) \geq 0$\\
$(ii)$ $S(x,y,z) = 0$ if and only if $x=y=z$,\\
$(iii)$  $S(x,y,z) \leq S(x,x,a) + S(y,y,a) + S(z,z,a) $\\
Then the pair $(X, S)$ is said to be a $S$-metric space.
\end{defi}
There are many examples of $S$-metric spaces have been thoroughly discussed in \cite{smetric}. There also it has been shown by an example that a $D^{*}$-metric is a $S$-metric but the converse is not true in general.\\
\indent Throughout $(X,S)$ or simply $X$ stands for a $S$-metric space, $\mathbb{R}$ for the set of all real numbers, $\mathbb{N}$ for the set of all natural numbers, sets are always subsets of $X$ unless otherwise stated. We should note that in a $S$-metric space $S(x,x,y)= S(y,y,x)$.
\begin{defi}\cite{smetric}
In a $S$-metric space  $(X,S)$ for $r>0$ and $x \in X$ we define open and closed ball of radius $r$ and center $x$ respectably as follows  :\\ 
$B_{S}(x,r)=\{ y \in X : S(y,y,x)<r  \}$ \\
$B_{S}[x,r]=\{ y \in X : S(y,y,x) \leq r  \}$ 
\end{defi}
There are several examples of open and closed sets in a $S$-metric space have been discussed in \cite{smetric}.
\begin{defi}\cite{smetric}
In a $S$-metric space $(X,S)$ a subset $A$ of $X$ is said to be a open set if for every $x\in A$ there exists a $r>0$  such that $B_{S}(x,r) \subset A$.
\end{defi}
\begin{defi}\cite{smetric}
In a $S$-metric space $(X,S)$ a subset $A$ of $A$ is said to be $S$-bounded if if there exists a $r>0$ such that $S(x,x,y)<r$ for every $x,y \in X$.
\end{defi}
\begin{defi}\cite{smetric}
In a $S$-metric space $(X,S)$ a sequence $\{x_{n} \}$ is said to be converges to $x$ in $X$ if for every $\epsilon > 0$ there exists a natural number $k$ such that $S(x_{n}, x_{n}, x)< \epsilon$ for every $n \geq k$.
\end{defi}
\begin{defi}\cite{smetric}
In a $S$-metric space $(X,S)$ a sequence $\{x_{n} \}$ is said to be a Cauchy sequence in $X$ if for every $\epsilon > 0$ there exists a natural number $k$ such that $S(x_{n}, x_{n}, x_{m})< \epsilon$ for every $n,m \geq k$.
\end{defi}
\begin{defi}\cite{smetric}
A $S$-metric space $(X,S)$ is said to be complete if every Cauchy sequence converges in this space. 
\end{defi}
In a $S$-metric space $(X,S)$ the collection of all open subsets $\tau$ in $X$ forms a topology on $X$ called topology induced by $S$-metric. The open ball $B_{S}(x,r)$ of radios $r$ and centre $x$ is a open set in $X$.
\section{\bf{Rough convergence of sequences in a $S$- metric space}}
\begin{defi}
A sequence $\{ x_{n} \}$ in a $S$-metric space $(X, S)$ is said to be $r$-convergent to $p$ if for every $\epsilon > 0$ there exists a natural number $k$ such that $S(x_{n}, x_{n}, p) <r + \epsilon $ holds for all $n \geq k$.
\end{defi}
\indent Usually we denote it by $x_{n} \stackrel{r}{\longrightarrow} p$ and $r$ is said to be the degree of roughness of rough convergence of $\left\{x_{n} \right\}$. It should be noted that when $r=0$ the rough convergence becomes the classical convergence of sequences in a $S$-metric space. If $\left\{x_{n} \right\}$ is $r$-convergent to $p$, then $p$ is said to be a $r$-limit point of $\left\{x_{n} \right\}$.  For some $r$ as defined above, the set of all $r$-limit points of a sequence $\left\{x_{n} \right\}$ is said to be the $r$-limit set of the sequence $\left\{x_{n} \right\}$ and we will denote it by $LIM^{r}x_{n}$. Therefore we can say $LIM^{r}x_{n}= \left\{x_{0} \in X : x_{n} \stackrel{r}{\longrightarrow} x_{0}\right\}$.We should note that $r$-limit point of a sequence $\left\{x_{n} \right\}$ may not be unique.\\
\indent Let $X=\mathbb{R}$, the set of all real numbers and $S: X^{3} \longrightarrow \mathbb{R}$ be defined by $S(x, y, z)= |x-z| + |y-z|$. It can be easily clarified that $(X, S)$ is a $S$-metric space. Now let us consider a sequence $\{x_{n}\} =\{ (-1)^{n}\frac{1}{2^{n}}\}$ in $X$. Clearly $S(x_{n}, x_{n}, p)= 2|x_{n}- p|$ for some $p \in \mathbb{R}$. \\
\indent If $p > 0$ then by considering $ \epsilon < 2|\frac{1}{4} - p|$ we can easily clarify that there exists infinitely many natural numbers $n$ for which $S(x_{n}, x_{n}, p) < \epsilon $ does not hold. Also if $p < 0$ then choosing $\epsilon < 2|- \frac{1}{2} -p|$ we can find infinitely many natural numbers $n$ for which $S(x_{n}, x_{n}, p) < \epsilon $ does not hold.\\ 
\indent But for $ p > 0$ if we take $r= (2|- \frac{1}{2}- p|) + 1$ then $S(x_{n}, x_{n}, p) <r + \epsilon $ holds for all $n$. For $ p < 0$ if we take $r= (2| \frac{1}{4}- p|) + 1$ then $S(x_{n}, x_{n}, p) <r + \epsilon $ holds for all $n$. \qed 
\newline \\
\indent The following example shows that a sequence is rough convergent in a $S$-metric space may not convergent in that space.\\ 
\noindent \textbf{Example 3.1.} Let us consider the $S$-metric space $(X, S)$ as defined above. Now let us consider the sequence $\{x_{n} \}$ where $x_{n}=(-1)^{n}$ for $n \in \mathbb{N}$. Clearly $\{x_{n} \}$ is not a convergent sequence in $X$. Because if $p \in \mathbb{R}$ then $S(x_{n}, x_{n}, p)=2|x_{n} -p|$. So $S(x_{n}, x_{n}, p)$ equals to either $2|1+p|$ or $2|1 - p|$. Hence if $M= Min\{2|1+p|, 2|1 - p|\}$ then for $\epsilon < M$ then there there exists infinitely many $n$ for which $S(x_{n}, x_{n}, p) < \epsilon $ does not hold and hence $\{x_{n} \}$ is not a convergent sequence in $X$. \\
\indent But if $r= Max\{2|1+p|, 2|1 - p|\}$ then for any $\epsilon  > 0$ we have $S(x_{n}, x_{n}, p) < r+\epsilon $ for all $n$. Hence $\{x_{n} \}$ is a $r$-convergent sequence and $r$-converges to $p$.\qed  \\
\indent We will use the similar kind of concept as was in the case of a metric space for diameter of a set in a $S$-metric space. For a subset $A$ of $X$ the diameter of $A$ is $sup \{ S(x,x,y) : x, y\in A\}$. $A$ is said to be of infinite diameter in the case when supremum is not finite. \\
\indent It has been discussed by Phu\cite{PHU} in a normed linear space that The diameter of a $r$-limit set of a $r$-convergent sequence $\{ x_{n} \}$ in a $S$-metric space $X$ is not greater then $2r$. Here we have found out a similar kind of property in a $S$-metric space as follows.\\ 
\begin{thm}
The diameter of a $r$- limit set of a $r$-convergent sequence $\{ x_{n} \}$ in a $S$-metric space $X$ is not greater then $3r$.
\end{thm}
\begin{proof}
We shall show that there does not exist elements $y,z \in LIM^{r} x_{n} $, such that $S(y,y,z)>3r$ holds. If possible let there exists elements $y,z \in LIM^{r} x_{n} $ such that $s>3r$ holds where $s= S(y,y,z)$. Let us consider $\epsilon > 0$ be arbitrarily preassigned. So we can find a $k_{1}, k_{2} \in \mathbb{N}$ for this $\epsilon$ such that $S(x_{n}, x_{n}, y)< r+ \epsilon$ for all $n\geq k_{1}$ and $S(x_{n}, x_{n}, z)< r+ \epsilon$ for all $n\geq k_{2}$. If we suppose $k=max\{k_{1},k_{2}\}$ then $S(x_{n}, x_{n}, y)< r+ \epsilon$ and $S(x_{n}, x_{n}, z)< r+ \epsilon$ both hold for all $n \geq k \dots (i)$ . Now for $n \geq k$ we can write $S(y,y, z)\leq S(y,y, x_{n})+S(y,y, x_{n})+ S(z,z, x_{n}) \dots (ii)$. Hence by using $(i)$ on $(ii)$ we can write $S(y,y,z)< 3(r+ \epsilon)$. Now if we consider $\epsilon = (\frac{s}{3}- r)$, where $(\frac{s}{3}- r) > 0$ then we have $S(y,y,z)< s$, which is a contradiction. Hence there can not have elements $y,z \in LIM^{r} x_{n} $, such that $S(y,y,z)>3r$ holds.
\end{proof}
\indent The following property of a $r$-convergent sequence have been discussed by Phu\cite{PHU} in a normed linear space. Here we have discussed it in a $S$-metric space.\\ 
\begin{thm}
A sequence $\{ x_{n} \}$ $r$-converges to $x$ in a $S$-metric space $X$ will imply $B_{S}[x,r]=LIM^{r}x_{n}$, when $\{ x_{n} \}$ converges to $x$.
\end{thm}
\begin{proof}
 First suppose that $\{ x_{n} \}$ converges to $x$. Let $y\in B_{S}[x,r]$ and let $\epsilon$ be any arbitrary preassigned positive real number. So there exists a \textit{natural} number $k$ such that $S(x_{n}, x_{n}, x)< \frac{\epsilon}{2}$ for all $n \geq k$ and we also have $S(x,x,y) \leq r$. Hence for $n\geq k$ we have $S(x_{n}, x_{n}, y) \leq S(x_{n}, x_{n}, x) + S(x_{n}, x_{n}, x) + S(y,y,x)$ and hence $S(x_{n}, x_{n}, y) < r+ \epsilon$ holds for every $n\geq k$. Therefore $y\in LIM^{r}x_{n}$ and hence $B_{S}[x,r] \subset LIM^{r}x_{n}$.\\
\indent Now let $z \in LIM^{r}x_{n}$ and $\epsilon$ be any arbitrary preassigned positive real number. Now we can find a $p_{1}, p_{2}\in \mathbb{N}$ such that $S(x_{n}, x_{n}, z)< r+ \frac{\epsilon}{3}$ for all $n \geq p_{1}$ and $S(x_{n}, x_{n}, x)< \frac{\epsilon}{3}$ for all $n\geq p_{2}$. Now if $p= max\{p_{1}, p_{2}\}$ then $S(x_{n}, x_{n}, z)< r+ \frac{\epsilon}{3}$ and $S(x_{n}, x_{n}, x)< \frac{\epsilon}{3}$ holds for every $n\geq p \dots (i)$. Now for $n\geq p$, $S(x,x,z)\leq S(x,x, x_{n})+ S(x,x, x_{n})+ S(z,z, x_{n})= S( x_{n}, x_{n}, x)+ S( x_{n}, x_{n}, x)+ S( x_{n}, x_{n}, z) \dots (ii)$ holds. Hence by using $(i)$ on $(ii)$ we have $S(x,x,z)< r+ \epsilon$ for any $\epsilon > 0$ and hence $S(x,x,z)\leq r$ holds. Therefore $z\in B_{S}[x,r]$.
\end{proof}
\noindent \textbf{Remark :} If $x, y$ be two distinct elements in $X$ and $S(x,x,y) =d$ then we can have two disjoint open balls $B_{S}(x, \frac{d}{3})$ and $B_{S}(y, \frac{d}{3})$. Because if $z \in B_{S}(x, \frac{d}{3}) \cap B_{S}(y, \frac{d}{3})$ then $S(x,x,z) < \frac{d}{3}$ and $S(y,y,z) < \frac{d}{3}$. So $S(x,x,y) \leq S(x,x,z)+ S(x,x,z)+S(y,y,z) < d$. Which is a contradiction and hence every $S$-metric space is Hausdorff. Therefore every cnvergent sequence in a $S$-metric space has a unique limit.\\
\begin{lem}
Every closed set in $(X, S)$ contains all its limit points.
\end{lem}
Proof is straightforward and so is omitted.
\begin{lem}
In $(X,S)$ a subset $A$ of $X$ is closed if and only if every sequence in $A$ which converges in $X$ converges to a point in $A$.
\end{lem}
\begin{proof}
First suppose that every sequence in $A$ which converges in $X$ converges to a point in $A$. Now if possible let $A$ is not closed. Then there exists a limit point $p$ of $A$ in the complement of $A$. So $B_{S}(p, \frac{1}{n}) \cap (A\setminus \{p\}) \neq \phi$ for all $n \in \mathbb{N}$. Now if we consider an element $x_{n}$ from $B_{S}(p, \frac{1}{n}) \cap (A\setminus \{p\}) $  for all $n \in \mathbb{N}$ then $\{ x_{n} \}$ is a sequence in $A$ which converges to a point $p$ in the complement of $A$.\\
\indent Let $M$ be an open set containing $p$ then there exists an open ball $B_{S}(p, g)$ for some $g>0$ such that $p \in B_{S}(p, g) \subset M$ and hence there exists $h \in \mathbb{N}$ such that $p \in B_{S}(p, \frac{1}{h}) \subset B_{S}(p, g)$. Therefore for $n \geq h$ we have $x_{n} \in M$. This is a contradiction to our supposition.\\ 
\indent Conversely let $A$ is a closed set in $X$. Now if we suppose that there exists a sequence $\{x_{n} \}$ in $A$ converges to a point $r$ which is not in $A$. Then for every open ball B centered at $r$ we have $B \cap (A\setminus \{r\}) \neq \phi$, which is a contradiction to the fact that $A$ is closed in $X$. 
\end{proof}
The following result has been discussed by Phu\cite{PHU} in a normed linear space and is very important in this sequel. 
\begin{thm}
Let $\{ x_{n} \}$ be a sequence in $(X,S)$. Then $LIM^{r}x_{n}$ is a closed set for any degree of roughness $r \geq 0$.
\end{thm}
\begin{proof}
We have to only discuss the case when $LIM^{r}x_{n} (\neq \phi )$. Let $ \{y_{n}\}$ be a sequence in $LIM^{r}x_{n}(\neq \phi )$ converges to $y$. We will show that $ y \in LIM^{r}x_{n}$. Let $\epsilon > 0$ be arbitrary positive real number. Now we can find a $k \in \mathbb{N}$ such that \begin{center} $S(y_{n}, y_{n}, y) < \frac{\epsilon}{3}$ and $S(x_{n}, x_{n}, y_{p})< r+ \frac{\epsilon}{3}$ holds for every $n \geq k$.$ \dots (1)$\end{center} where $p$  is a fixed natural number greater then $k$.\\
\indent Now we can write $S(y,y,x_{n}) \leq S(y,y,y_{p})+ S(y,y,y_{p})+ S(x_{n}, x_{n},y_{p})=S(y_{p},y_{p}, y)+ S(y_{p},y_{p}, y)+ S(x_{n}, x_{n},y_{p}) \dots (2)$. Therefore using $(1)$ on $(2)$ for $n\geq k$ we can write $S(y,y,x_{n})=S(x_{n},x_{n}, y)<r+ \epsilon$. Hence $ y \in LIM^{r}x_{n}$ and hence $LIM^{r}x_{n}$ is a closed set in $X$.
\end{proof}
\indent In $(X, S)$ a sequence $\{ x_{n}\}$ is said to be bounded if and only if there exists a $B(>0)\in \mathbb{R}$ such that $S(x_{n}, x_{n}, x_{m}) < B$ for all $m,n \in \mathbb{N}$. The following theorem is a generalization of the classical properly of a sequence that a convergent sequence must be bounded. 
\begin{thm}
Every $r$-convergent sequence in a $S$-metric space is bounded.
\end{thm}
\begin{proof}
Let $\{ x_{n}\}$ be a $r$-convergent sequence in a $s$-metric space $(X, s)$ and $r$-converges to $x$. We will show that $\{ x_{n}\}$ is bounded in $X$. Now for any arbitrary $\epsilon > 0$ we can find a natural number $k$ such that $S(x_{n}, x_{n}, x)< r+ \epsilon$ for all $n \geq k$. Now consider $L = max_{1\leq i,j \leq k} \{ S(x_{i}, x_{i}, x_{j}) \}$.\\  
\indent Now let $i\leq k$ and $j\geq k$ then then $ S(x_{j}, x_{j}, x_{k}) \leq S(x_{j}, x_{j}, x)+ S(x_{j}, x_{j}, x)+ S(x_{k}, x_{k}, x) <3(r+ \epsilon)$. Also $S(x_{i}, x_{i}, x_{j})\leq S(x_{i}, x_{i}, x_{k}) +S(x_{i}, x_{i}, x_{k}) + S(x_{j}, x_{j}, x_{k}) <2L + 3(r+ \epsilon)$. If $i\geq k$ and $j\leq k$ similarly we can show that $S(x_{i}, x_{i}, x_{j})< 6(r+ \epsilon) + L$.\\ 
\indent Now consider the case for $i\geq k$ and $j\geq k$ we have $S(x_{i}, x_{i}, x_{j})\leq S(x_{i}, x_{i}, x) +S(x_{i}, x_{i}, x) + S(x_{j}, x_{j}, x) < 3(r+\epsilon)$.\\
\indent Now if $B> max \{ 3(r+ \epsilon), L, (2L + 3r + 3 \epsilon), (6r + 6 \epsilon +L)\}$ then $S(x_{i}, x_{i}, x_{j})< B$ for all $i,j\in \mathbb{N}$. Therefore $\{ x_{n}\}$ is bounded in $X$  and hence the result follows.
\end{proof}
\indent We know that a bounded sequence in a metric space may not be convergent. But it has been studied in \cite{RMROUGH} that a bounded sequence in a cone metric space is rough convergent for some degree of roughness. In the next theorem we will extend this idea in a $s$-metric space.
\begin {thm}
A bounded sequence in a $S$-metric space is always $r$-convergent for some degree of roughness $r$.
\end{thm}
\begin{proof}
Let $\{ x_{n}\}$ be a bounded sequence in a cone metric space $(X, S)$. So there exists a $B>0$ such that $S(x_{i}, x_{i}, x_{j}) < B$ for all $i,j \in \mathbb{N}$. Hence for any $\epsilon >0$ we have $S(x_{i}, x_{i}, x_{j}) < B + \epsilon$ for all $i,j$. Hence $\{ x_{n}\}$ rough converges to $x_{p}$ for every $p\in \mathbb{N}$ for degree of roughness $B$ and hence the result follows.
\end{proof}
\begin {thm}
Let $\{a_{n}\}$ and $\{b_{n}\}$ be two sequences in a $S$-metric space $(X, S)$ with the property that $S(a_{i}, a_{i}, b_{i})\leq \frac{r}{2}$ for all $i \geq k_{1}( \in \mathbb{N})$ and $r> 0$. Then if $\{ a_{n} \}$ converges to a $\xi \in X$ will imply $\{b_{n}\}$ is $r$- convergent to $\xi$.
\end{thm}
\begin{proof}
Let $\epsilon > 0$ be a preassigned real number. Now since $\{a_{n}\}$ converges to $\xi$ corresponding to this $\epsilon > 0$ we can find a natural number $k_{2}$ such that $S(a_{n}, a_{n}, \xi)< \epsilon $ for all $n \geq k_{2}$.\\
\indent Now if we consider $k=max(k_{1}, k_{2})$ then for all $n \geq k$ we have $S(b_{n}, b_{n}, \xi) \leq S(b_{n}, b_{n}, a_{n}) + S(b_{n}, b_{n}, a_{n})+ S(\xi, \xi, a_{n})< \frac{r}{2}+ \frac{r}{2} + \epsilon$. Because $S(b_{n}, b_{n}, a_{n}) = S(a_{n}, a_{n}, b_{n}) \leq \frac{r}{2}$ and $S(\xi, \xi, a_{n})= S(a_{n}, a_{n}, \xi) < \epsilon$ for all $n\geq k$. Hence the result follows.
\end{proof}
\indent Let $\{x_{n}\}$ is a $r$-convergent sequence in a $S$-metric space $(X, S)$ and $r$-converges to $x$. Then it have been already discussed in \cite{RMROUGH, PHU} that if $\{\xi_{n}\}$ be a convergent sequence in $LIM^{r}x_{n}$ and converges to $\xi$ will imply $\xi \in LIM^{r}x_{n}$. We have found out similar kind of property in a $S$-metric space as discussed bellow.
\begin {thm}
Let $\{x_{n}\}$ is a $r$-convergent sequence in a $S$-metric space $(X, S)$ and $r$-converges to $x$. Then if $\{\xi_{n}\}$ be a convergent sequence in $LIM^{r}x_{n}$ and converges to $\xi$ will imply $\{x_{n}\}$ is a $2r$-convergent sequence in $X$ and $2r$-converges to $\xi$.
\end{thm}
\begin{proof}
Let $\epsilon > 0$ be a preassigned real number. Since $\{\xi_{n}\}$ converges to $\xi$ there exists a natural number $k_{1} $ such that $S(\xi_{n}, \xi_{n}, \xi) < \frac{\epsilon}{3}$ for all $n \geq k_{1}$. Also since $\{x_{n}\}$ is a $r$-convergent to $x$ there exists a natural number $k_{2}$ such that $S(x_{n}, x_{n}, x)< r+ \frac{\epsilon}{3}$ for all $n \geq k_{2}$. Let $k$ be the maximum of $k_{1}$ and $k_{2}$ and consider a member $\xi_{m}$ of $\{\xi_{n}\}$ where $m> k$.\\
\indent  Now for all $n\geq k$ we have $S(x_{n}, x_{n}, \xi) \leq S(x_{n}, x_{n}, \xi_{m})+S(x_{n}, x_{n}, \xi_{m})+ S(\xi, \xi, \xi_{m}) =S(x_{n}, x_{n}, \xi_{m})+ S(x_{n}, x_{n}, \xi_{m})+ S(\xi_{m}, \xi_{m}, \xi) <(r+ \frac{\epsilon}{3})+(r+ \frac{\epsilon}{3})+ \frac{\epsilon}{3}= 2r+ \epsilon$. Therefore $S(x_{n}, x_{n}, \xi)< 2r+ \epsilon$ holds for all $n\geq k$ and hence the result follows.
\end{proof}
\vspace{0.2in}
\indent Let $\{x_{n}\}$ be a sequence in a $S$-metric space $(X, S)$ then a $\xi \in X$ is said to be a cluster point of $\{x_{n}\}$ if for every $\epsilon >0$ and every natural number $p$ there exists a natural number $k>p$ such that $s(x_{k}, x_{k}, \xi)< \epsilon$ holds.   The idea of a closed ball has been discussed previously.\\
\indent The following result has been discussed in a cone metric space\cite{RMROUGH} here we have verified the same in a $S$-metric space. It is a characterization of rough convergence in terms of cluster points.
\begin {thm}
Let $\{x_{n}\}$ be a $r$-convergent sequence in a $S$-metric space $(X, S)$. Then for any cluster point $c$ of $\{x_{n}\}$, $LIM^{r}x_{n}  \subset B_{S}[c,r]$ holds.
\end{thm}
\begin{proof}
Let $\epsilon$ be a preassigned positive quantity and $x \in LIM^{r}x_{n}$. We can find a natural number $k$ such that $S(x_{n}, x_{n}, x)< \frac{\epsilon}{3} +r$ for all $n\geq k$. Also since $c$ is a cluster point of $\{x_{n}\}$ there exists a natural number $m>k$ such that $S(x_{m}, x_{m}, c)< \frac{\epsilon}{3}$ holds.\\
\indent Now for this natural number $m$ we can write $S(c,c,x) \leq S(c,c,x_{m})+ S(c,c,x_{m})+ S(x,x,x_{m})=S(x_{m},x_{m}, c)+ S(x_{m},x_{m}, c)+ S(x_{m},x_{m}, x) <\frac{\epsilon}{3}+ \frac{\epsilon}{3} +(r+\frac{\epsilon}{3})=r+ \epsilon$. Hence $S(c,c,x)<r+ \epsilon$. Since $\epsilon$ is chosen arbitrarily  $S(c,c,x)\leq r$ and hence $S(x,x,c)\leq r$. Therefore $x \in B_{S}[c,r]$ and the result follows.
\end{proof}
\noindent \textbf{Acknowledgement:}
The authors are grateful to Prof. Amar Kumar Banerjee, Department of Mathematics, University of Burdwan for his advice during the preparation of this paper. The second author is thankful to The University of Burdwan for the grant of Junior Research Fellowship (State Funded) during the preparation of this paper.

\end{document}